\documentclass[12pt]{amsart}
\usepackage{amsmath,amssymb}
\usepackage{amsfonts}
\usepackage{amsthm}
\usepackage{latexsym}
\usepackage{graphicx}

\def\p{\partial}

\def\R{\mathbb{R}}

\def\vv<#1>{\langle#1\rangle}

\def\XXint#1#2{\setbox0=\hbox{$#1{#2}{\int}$}{#2}\kern-.5\wd0 }

\def\XXint#1#2#3{{\setbox0=\hbox{$#1{#2#3}{\int}$}
     \vcenter{\hbox{$#2#3$}}\kern-.5\wd0}}

\def\vv<#1>{\langle#1\rangle}
\newtheorem{thm}{Theorem}[section]

\newtheorem{cor}{Corollary}[section]
\theoremstyle{definition}

\theoremstyle{remark}

\numberwithin{equation}{section}

\begin{document}

\title{A note on Li-Yau type gradient estimate}

\author{Chengjie Yu$^1$}
\address{Department of Mathematics, Shantou University, Shantou, Guangdong, 515063, China}
\email{cjyu@stu.edu.cn}
\author{Feifei Zhao}
\address{Department of Mathematics, Shantou University, Shantou, Guangdong, 515063, China}
\email{14ffzhao@stu.edu.cn}
\thanks{$^1$Research partially supported by the Yangfan project from Guangdong Province and NSFC 11571215.}

\renewcommand{\subjclassname}{%
  \textup{2010} Mathematics Subject Classification}
\date{}
\begin{abstract}
In this paper, we obtain  Li-Yau type gradient estimates with time dependent parameter for positive solutions of the heat equation that are different with the estimates by Li-Xu \cite{LX} and Qian \cite{Qi}. As an application of the estimate, we also obtained slight improvements of Davies' Li-Yau type gradient estimate.
\end{abstract}
\maketitle\markboth{Yu \& Zhao}{Li-Yau type gradient estimate}
\section{Introduction}
Let $(M^n,g)$ be a complete Riemannian manifold with Ricci curvature bounded from below by $-k$, where $k$ is a nonnegative constant. Let $u$ be a positive solution of the heat equation:
\begin{equation}\label{eq-heat}
\Delta u-u_t=0.
\end{equation}
In the fundamental paper \cite{LY}, Li and Yau obtained the following important gradient estimate for $u$:
\begin{equation}\label{eq-LY}
\|\nabla f\|^2-\alpha f_t\leq \frac{n\alpha^2}{2t}+\frac{n\alpha^2k}{2(\alpha-1)}
\end{equation}
for any $\alpha>1$, where $f=\log u$. When $k=0$, by letting $\alpha\to1$, one have
\begin{equation}\label{eq-LY-0}
\|\nabla f\|^2-f_t\leq \frac{n}{2t}.
\end{equation}
This estimate is sharp where the equality can be achieved by the fundamental solution of  $\R^n$. However, \eqref{eq-LY} is not sharp when $k>0$. Finding sharp Li-Yau type gradient estimate for $k>0$ is still an open problem.

The Li-Yau gradient estimate \eqref{eq-LY} is an important tool in geometric analysis. It gives us the Harnack inequality for positive solution of the heat equation and Gaussian bounds for the heat kernel. Many works were done to improve or generalize \eqref{eq-LY}, for examples, the works \cite{BQ,BL,CFL,Ya94,Ya95}. Recently, in \cite{ZZh}, the authors extended the Li-Yau type gradient estimate to metric measure spaces, and in \cite{ZZ0,ZZ,Ca,Ro}, the authors obtained Li-Yau type gradient estimate under integral curvature assumptions. The Li-Yau gradient estimate was also extended to a matrix form by Hamilton \cite{Ha}, extended to complete K\"ahler manifolds by Cao and Ni \cite{CN},extended to Ricci flow and K\"ahler-Ricci flow by Hamilton \cite{Ha} and Cao \cite{Cao} respectively, extended to Hodge heat flows of $(p,p)$-forms on K\"ahler manifolds by Ni and Niu in \cite{NN}, and extended to the constraint case  in \cite{CH} and \cite{RY}. Li-Yau type gradient estimates also played important roles in Perelman's  work \cite{Pe}.

Li-Yau's gradient estimate \eqref{eq-LY} was improved by Davies \cite{Da} to
\begin{equation}\label{eq-Davies}
\|\nabla f\|^2-\alpha f_t\leq \frac{n\alpha^2}{2t}+\frac{n\alpha^2k}{4(\alpha-1)}.
\end{equation}
By comparing to the fundamental solution, it is clear that \eqref{eq-Davies} is even not sharp in leading term as $t\to 0$. In \cite{Ha}, Hamilton obtained
\begin{equation}\label{eq-Ha}
\|\nabla f\|^2-e^{2kt}f_t\leq e^{4kt}\frac{n}{2t},
\end{equation}
which is sharp in leading term as $t\to 0$. However, the estimate behaves bad when $t\to\infty$.
In recent years, Li and Xu \cite{LX} obtained
\begin{equation}\label{eq-LX}
\|\nabla f\|^2-\left(1+\frac{\sinh(kt)\cosh(kt)-kt}{\sinh^{2}(kt)}\right)f_t\leq\frac{nk}{2}[\coth(kt)+1].
\end{equation}
Note that
\begin{equation}
1+\frac{\sinh(kt)\cosh(kt)-kt}{\sinh^{2}(kt)}\sim 1
\end{equation}
and
\begin{equation}
\frac{nk}{2}[\coth(kt)+1]\sim \frac{n}{2t}
\end{equation}
as $t\to 0^+$. Hence, \eqref{eq-LX} is sharp in leading term as $t\to 0$. Furthermore, note that
\begin{equation}
1+\frac{\sinh(kt)\cosh(kt)-kt}{\sinh^{2}(kt)}\sim 2 (\triangleq\alpha)
\end{equation}
and
\begin{equation}
\frac{nk}{2}[\coth(kt)+1]\sim nk=\frac{n\alpha^2k}{4(\alpha-1)}
\end{equation}
as $t\to\infty$. The asymptotic behavior of \eqref{eq-LX} as $t\to\infty$ is the same as \eqref{eq-Davies} with $\alpha=2$ , and is better than \eqref{eq-Ha}. This estimate was also obtained in \cite{BB} by a different method. In the same paper \cite{LX}, Li and Xu  also obtained a linearized version of \eqref{eq-LX}:
\begin{equation}\label{eq-LX-linear}
\|\nabla f\|^2-\left(1+\frac{2}{3}kt\right)f_t\leq\frac{n}{2t}+\frac{nk}{2}\left(1+\frac{1}{3}kt\right).
\end{equation}
This estimate is also sharp in leading term as $t\to 0$. Although blowing up linearly as $t\to\infty$, it is still better than \eqref{eq-Ha} as $t\to\infty$.
This estimate was previously obtained by Bakry and Qian \cite{BQ} with a different method.

The estimates \eqref{eq-LX} and \eqref{eq-LX-linear} were later generalized by Qian \cite{Qi} to the following general form:
\begin{equation}\label{eq-Qian}
\|\nabla f\|^2-\left(1+\frac{2k}{a(t)}\int_0^ta(s)ds\right) f_t\leq \frac{nk}{2}+\frac{nk^{2}}{2a(t)}\int_{0}^{t}a(s)ds+\frac{n}{8a(t)}\int_{0}^{t}\frac{a'^{2}(s)}{a(s)}ds,
\end{equation}
where  $a(t)$ is a smooth function satisfying:
\begin{enumerate}
 \item[(A1)]$\forall t>0$, $a(t)>0$, $a'(t)>0$;
 \item[(A2)]$\lim_{t\to0}a(t)=0$, $\lim_{t\to0}\frac{a(t)}{a'(t)}=0$;
 \item[(A3)] $\frac{a'^2}{a}$ is integrable near $0$.
 \end{enumerate}
 The estimates \eqref{eq-LX} and \eqref{eq-LX-linear} are special cases of \eqref{eq-Qian} with $a(t)=\sinh^2(kt)$ and $a(t)=t^2$ respectively. Moreover, by choosing $a(t)=t^{\frac{2}\theta-1}$ with $\theta\in (0,1)$, Qian \cite{Qi} obtained the following extension of \eqref{eq-LX-linear}:
 \begin{equation}\label{eq-Qian2}
\|\nabla f\|^2-(1+\theta kt) f_t\leq\frac{(2-\theta)^2n}{16\theta(1-\theta)t}+\frac{nk^2\theta t}{4}+\frac{nk}{2}.
\end{equation}
When $\theta=\frac{2}{3}$ which is the minimum point of $\frac{(2-\theta)^2}{\theta(1-\theta)}$ with $\theta\in (0,1)$, it gives us \eqref{eq-LX-linear}.

For convenience of comparison, one can rewrite Davies' Li-Yau gradient estimate as
\begin{equation}\label{eq-Davies2}
\beta\|\nabla f\|^2-f_t\leq \frac{n}{2\beta t}+\frac{nk}{4(1-\beta)}
\end{equation}
for any $\beta\in (0,1)$. For example, for any fixed $t>0$, the right hand side of \eqref{eq-Davies2} achieves its minimum at
\begin{equation}
\beta_m(t)=\frac{1}{1+\sqrt\frac{kt}2}.
\end{equation}
Therefore, for a fixed $t>0$, \eqref{eq-Davies2} with $\beta<\beta_m(t)$ can be implied by \eqref{eq-Davies2} with $\beta=\beta_m(t)$ since
\begin{equation}
\begin{split}
\beta\|\nabla f\|^2-f_t\leq&\beta_m(t)\|\nabla f\|^2-f_t\\
\leq&\frac{n}{2\beta_m(t) t}+\frac{nk}{4(1-\beta_m(t))}\\
\leq&\frac{n}{2\beta t}+\frac{nk}{4(1-\beta)}\\
\end{split}
\end{equation}
in this case. In \cite{ZZ0, ZZ}, the authors also wrote the Li-Yau gradient estimate in this form.

In this paper, we  first obtain the following Li-Yau type gradient estimate with time dependent parameter.
\begin{thm}\label{thm-LY-g}
Let $(M^n,g)$ be a complete Riemannian manifold with Ricci curvature bounded from below by $-k$, where $k$ is a nonnegative constant. Let $u\in C^\infty(M\times [0,T])$ be a positive solution of the heat equation \eqref{eq-heat} and $\beta\in C^1([0,T])$ such that
\begin{enumerate}
\item[(B1)] $0<\beta(t)<1$ for any $t\in (0,T]$;
\item[(B2)] $(1-\beta(0))^2+\beta'(0)^2>0$ and $\beta(0)>0$.
\end{enumerate}
Let
\begin{equation}
\psi_1(t)=\frac{n}{2t}\max_{s\in [0,t]}\left(\frac{1}{\beta(s)}+\frac{(2k\beta(s)+\beta'(s))_+s}{4\beta(s)(1-\beta(s))}\right).
\end{equation}
Then,
\begin{equation}
\beta(t)\|\nabla f\|^2-f_t\leq\psi_1(t)
\end{equation}
on $M\times(0,T]$, where $f=\log u$. Here $a_+:=\max\{a,0\}$.
\end{thm}
When $\beta$ is constant, Theorem \ref{thm-LY-g} gives us \eqref{eq-Davies2}. Moreover, by choosing $\beta(t)=e^{-2\theta kt}$ with $\theta\in (0,1]$ in Theorem \ref{thm-LY-g}, we obtain
\begin{equation}\label{eq-Ha-g}
e^{-2\theta kt}\|\nabla f\|^2-f_t\leq \frac{n}{2t}e^{2\theta kt}+\frac{nk(1-\theta)}{4(1-e^{-2\theta kt})}.
\end{equation}
When $\theta=1$, this gives us Hamilton's Li-Yau type gradient estimate \eqref{eq-Ha}. By choosing $\beta(t)=\frac{1}{1+\theta kt}$ with $\theta>0$ in Theorem \ref{thm-LY-g}, one has
\begin{equation}\label{eq-LY-linear}
  \frac{1}{1+\theta kt}\|\nabla f\|^2-f_t\leq \frac{n(1+\theta kt)}{2t}+\frac{n(2-\theta+2\theta kt)_+}{8\theta t}.
\end{equation}
Comparing to \eqref{eq-Qian2}, this estimate has an advantage that the right hand side does not blow up when $\theta\to1$ and works for any $\theta>0$.

Moreover, Theorem \ref{thm-LY-g} can give us improvements of the Li-Yau type gradient estimate \eqref{eq-Davies2}. For any $t_0\in (0,T]$ and $\beta_0\in (0,1)$, let
\begin{equation}\label{eq-phi}
\begin{split}
\varphi_1(\beta_0,t_0)=\inf\{&\psi_1(t_0)\ |\\
\ \ \ \ \ \ \ \ \ \ \ \ & \beta\ \mbox{satisfies (B1), (B2) on }[0,t_0]\mbox{, and }\beta(t_0)=\beta_0. \}
\end{split}
\end{equation}
Then, by Theorem \ref{thm-LY-g},
\begin{equation}
\beta\|\nabla f\|^2-f_t\leq\varphi_1(\beta,t).
\end{equation}
So, a good upper bound of the function $\varphi_1(\beta, t)$ will give us an improvement of \eqref{eq-Davies2}. For example, by choosing $\beta(t)=1-\theta kt$
with $\theta>0$ as a test function, one can obtain the following improvement of \eqref{eq-Davies2}.
\begin{cor}\label{cor-LY-g}
Let the notation be the same as in Theorem \ref{thm-LY-g} with $k>0$. Then, for any constant $\beta\in (0,1)$, we have
\begin{equation}\label{eq-cor-LY-g}
\beta\|\nabla f\|^2-f_t\leq \left\{\begin{array}{ll} \frac{n}{2\beta t}&t< \frac{1-\beta}{2k\beta }\\
\frac{3n}{8\beta t}+\frac{nk}{4(1-\beta)}&t\geq \frac{1-\beta}{2k\beta }\end{array}\right.
\end{equation}
on $M\times (0,T]$.
\end{cor}

Furthermore, by a slightly different argument with that in the proof of Theorem \ref{thm-LY-g}, we have the following different Li-Yau type estimate with time dependent parameter.
\begin{thm}\label{thm-LY-g-2}
Let the notation be the same as in Theorem \ref{thm-LY-g},
\begin{equation}
\sigma(t)=\max_{s\in[0,t]}\frac{(2k\beta(s)+\beta'(s))_+s}{1-\beta(s)}
\end{equation}
and
\begin{equation}
\lambda(t)=\inf_{s\in [0,t]}\beta(s).
\end{equation}
Then,
\begin{equation}
\beta(t)\|\nabla f\|^2-f_t\leq\psi_2(t)
\end{equation}
on $M\times(0,T]$, where
\begin{equation}
\psi_2(t)=\left\{\begin{array}{ll}\frac{n}{2\lambda t}&\sigma(t)< 2\\
\frac{n\sigma^2}{8(\sigma-1)\lambda t}&\sigma(t)\geq 2.\end{array}\right.
\end{equation}
\end{thm}
When $\beta$ is constant, Theorem \ref{thm-LY-g-2} gives us the following improvement of  \eqref{eq-Davies2}.
\begin{cor}\label{cor-LY-g-2}
Let the notation be the same as in Theorem \ref{thm-LY-g} with $k>0$. Then, for any constant $\beta\in (0,1)$, we have
\begin{equation}\label{eq-cor-LY-g-2}
\beta\|\nabla f\|^2-f_t\leq \left\{\begin{array}{ll} \frac{n}{2\beta t}&t< \frac{1-\beta}{k\beta }\\
\frac{n}{4\beta t}+\frac{nk}{4(1-\beta)}&t\geq \frac{1-\beta}{k\beta }\end{array}\right.
\end{equation}
on $M\times (0,T]$.
\end{cor}
It is clear that \eqref{eq-cor-LY-g-2} is better than \eqref{eq-cor-LY-g}. If we choose $\beta(t)=e^{-2kt}$, Theorem \ref{thm-LY-g-2} also gives us Hamilton's Li-Yau type gradient estimate \eqref{eq-Ha}. Moreover, if we choose $\beta(t)=e^{-2\theta kt}$ with $\theta\in (0,1)$ and $\beta(t)=\frac{1}{1+\theta kt}$, it is not hard to check that Theorem \ref{thm-LY-g-2} gives us improvements of \eqref{eq-Ha-g} and \eqref{eq-LY-linear} respectively.

Similar as before,
for any $t_0\in (0,T]$ and $\beta_0\in (0,1)$, let
\begin{equation}\label{eq-phi-2}
\begin{split}
\varphi_2(\beta_0,t_0)=\inf\{&\psi_2(t_0)\ |\\
\ \ \ \ \ \ \ \ \ \ \ \ & \beta\ \mbox{satisfies (B1), (B2) on }[0,t_0]\mbox{, and }\beta(t_0)=\beta_0. \}
\end{split}
\end{equation}
Then, by Theorem \ref{thm-LY-g-2},
\begin{equation}
\beta\|\nabla f\|^2-f_t\leq\varphi_2(\beta,t).
\end{equation}
By using $\beta(t)=1-\theta kt$ as a test function, we have the following different Li-Yau type gradient estimate.
\begin{cor}\label{cor-LY-g-3}
Let the notation be the same as in Theorem \ref{thm-LY-g} with $k>0$. Then, for any constant $\beta\in (0,1)$, we have
\begin{equation}\label{eq-cor-LY-g-3}
\beta\|\nabla f\|^2-f_t\leq \left\{\begin{array}{ll} \frac{n}{2\beta t}&t< \frac{3(1-\beta)}{2k}\\
\frac{3(1-\beta)n}{16k\beta t^2}+\frac{nk}{4(1-\beta)\beta}&t\geq \frac{3(1-\beta)}{2k}\end{array}\right.
\end{equation}
on $M\times (0,T]$.
\end{cor}
This estimate is weaker than \eqref{eq-Davies2} when $t$ tends to infinity. From this, one can see that Theorem \ref{thm-LY-g} produces stronger estimates than that of Theorem \ref{thm-LY-g-2} in some cases.

Finally, we would like to remark that although the leading term $\frac{n}{2\beta t}$ of the Li-Yau type estimates for small time in \eqref{eq-cor-LY-g}, \eqref{eq-cor-LY-g-2} is not sharp comparing to \eqref{eq-Ha}, \eqref{eq-LX} and \eqref{eq-LX-linear} where the leading term is $\frac{n}{2t}$, and the asymptotic behaviors as $t\to \infty$ of  $\eqref{eq-cor-LY-g}, \eqref{eq-cor-LY-g-2}$ are the same as those of \eqref{eq-Davies}, \eqref{eq-LX}, the estimates \eqref{eq-cor-LY-g}, \eqref{eq-cor-LY-g-2} give slight improvements of Davies' Li-Yau type estimate \eqref{eq-Davies} for all time. The comparison of \eqref{eq-LX} and \eqref{eq-Davies} for all time is not quite clear.
\section{Li-Yau type gradient estimate}
We first prove Theorem \ref{thm-LY-g}.

\begin{proof}[Proof of Theorem \ref{thm-LY-g}] Because the proof of the compact case is similar and simpler, we only prove the complete noncompact case.

Let $F=t(\beta\|\nabla f\|^2-f_t)$, and $L=\Delta-\p_t$. Then,
\begin{equation}\label{eq-f}
Lf=-\|\nabla f\|^2,
\end{equation}
\begin{equation}\label{eq-ft}
Lf_t=-(\|\nabla f\|^2)_t=-2\vv<\nabla f_t,\nabla f>,
\end{equation}
and
\begin{equation}\label{eq-grad-f}
\begin{split}
&L(\|\nabla f\|^2)\\
=&2\|\nabla^2f\|^2+2\vv<\nabla\Delta f,\nabla f>+2Ric(\nabla f,\nabla f)-2\vv<\nabla f_t,\nabla f>\\
\geq&\frac{2}{n}(\Delta f)^2-2k\|\nabla f\|^2-2\vv<\nabla\|\nabla f\|^2,\nabla f>.
\end{split}
\end{equation}
Then, by \eqref{eq-ft} and \eqref{eq-grad-f},
\begin{equation}\label{eq-F}
\begin{split}
&LF\\
\geq&\frac{2\beta t}{n}(\Delta f)^2-(2k\beta+\beta') t\|\nabla f\|^2-2\vv<\nabla F,\nabla f>-\frac{F}{t}\\
=&\frac{2\beta t}{n}(\|\nabla f\|^2-f_t)^2-(2k\beta+\beta') t\|\nabla f\|^2-2\vv<\nabla F,\nabla f>-\frac{F}{t}\\
=&\frac{2\beta t}{n}\left(\frac{F}{ t}+\left(1-\beta\right)\|\nabla f\|^2\right)^2-(2k\beta+\beta') t\|\nabla f\|^2-2\vv<\nabla F,\nabla f>-\frac{F}{t}.\\
\end{split}
\end{equation}

Let $\eta$ be a smooth function on $[0,\infty)$ such that
\begin{enumerate}
\item $0\leq\eta\leq 1$;
\item $\eta'\leq 0$;
\item $\eta(t)=1$ on $t\in [0,1]$;
\item $\eta(t)=0$ on $t\in [2,\infty)$.
\end{enumerate}
Let $\rho=\eta^2(r(x)/R)$ where $r(x)=r(x,p)$ where $p$ is a fixed point in $M$. Then, by the Laplacian comparison theorem,
\begin{equation}\label{eq-grad-rho}
\|\nabla \rho\|^2\leq C_1 R^{-2}\rho
\end{equation}
and
\begin{equation}\label{eq-Lap-rho}
\Delta\rho\geq -C_1R^{-1}
\end{equation}
where $C_1>1$  depends on $n$.

Let $G=\rho F$. For any $t>0$, let
$(x_0,t_0)$ be the maximum point of $G$ on $M\times [0,t]$. Since $G(x,0)=0$, we can assume that $G(x_0,t_0)>0$, and hence $x_0\in B_p(2R)$ and $t_0>0$. Moreover, by the Calabi trick (see \cite{LY}), we can assume that $x_0$ is not a cut point of $p$. So,
\begin{equation}
\nabla G(x_0,t_0)=0,
\end{equation}
and hence
\begin{equation}\label{eq-grad-F}
  \nabla F(x_0,t_0)=-F\rho^{-1}\nabla\rho,
\end{equation}
and
\begin{equation}
LG(x_0,t_0)\leq 0.
\end{equation}
So, by \eqref{eq-F}, at the point $(x_0,t_0)$, we have
\begin{equation}
\begin{split}
0\geq&LG(x_0,t_0)\\
=&\rho LF+F\Delta\rho+2\vv<\nabla F,\nabla\rho>\\
\geq &\frac{2\beta t}{n}\left(\frac{1}{ t}+\left(1-\beta\right)Q\right)^2\rho F^2-((2k\beta +\beta')_+tQ+1/t)\rho F\\
&-2\vv<\rho\nabla F,\nabla f>+2\vv<\nabla F,\nabla\rho>+F\Delta\rho,\\
\end{split}
\end{equation}
where $Q=F^{-1}\|\nabla f\|^2(x_0,t_0)$. Multiplying $t_0\rho(x_0)$  to the last inequality, and noting that $0\leq \rho\leq 1$, \eqref{eq-grad-rho}, \eqref{eq-Lap-rho} and \eqref{eq-grad-F}, we have, at the point $(x_0,t_0)$,
\begin{equation}\label{eq-G-2}
\begin{split}
0\geq&\frac{2\beta }{n}\left(1+\left(1-\beta\right)Qt\right)^2G^2-((2k\beta +\beta')_+t^2Q+1)G\\
&-2C_1R^{-1}Q^\frac12t G^\frac32-3C_1R^{-1}tG\\
\geq&\frac{2\beta }{n}\left(1+\left(1-\beta\right)Qt\right)^2G^2-((2k\beta +\beta')_+t^2Q+1)G\\
&-C_1R^{-1}G^2-C_1R^{-1}Qt^2 G-3C_1R^{-1}tG.\\
\end{split}
\end{equation}
By (B2), we know that $\min_{[0,T]}\beta>0$. Hence, when $R$ is large enough, we have
\begin{equation}
\frac{2\beta }{n}\left(1+\left(1-\beta\right)Qt_0\right)^2-C_1R^{-1}>0.
\end{equation}
Then, by \eqref{eq-G-2}, when $R$ is sufficiently large,
\begin{equation}\label{eq-G-3}
G(x_0,t_0)\leq \frac{n}{2\beta}\frac{[(2k\beta +\beta')_++C_1R^{-1}]t_0^2Q+3C_1R^{-1}t_0+1}{(1+(1-\beta)t_0Q)^2-C_2R^{-1}}
\end{equation}
where $C_2=\frac{nC_1}{2\min_{[0,T]}\beta}$. Note that
\begin{equation}
\frac{1}{a-\epsilon}\leq \frac{1+2\epsilon}{a}
\end{equation}
when $a\geq1$ and $\epsilon\leq1/2$. By \eqref{eq-G-3}, when $R$ is sufficiently large, we have
\begin{equation}\label{eq-G-4}
G(x_0,t_0)\leq \left(1+\frac{2C_2}{R}\right)\frac{n}{2\beta}\frac{[(2k\beta +\beta')_++C_1R^{-1}]t_0^2Q+3C_1R^{-1}t_0+1}{(1+(1-\beta)t_0Q)^2}.
\end{equation}
Moreover, note that
\begin{equation}
\frac{aQ+c}{(1+bQ)^2}\leq\frac{a}{4b}+c
\end{equation}
where $a,b,c,Q>0$, and by (B2) ,
\begin{equation}
\frac{t}{1-\beta}\leq C_3
\end{equation}
for $t\in [0,T]$. Applying this to \eqref{eq-G-4} and by (B2), we have
\begin{equation}
\begin{split}
G(x,t)\leq &G(x_0,t_0)\\
\leq& \left(1+\frac{2C_2}{R}\right)\frac{n}{2\beta}\left(\frac{(2k\beta+\beta')_+t_0}{4(1-\beta)}+C_4R^{-1}+3C_1R^{-1}t_0+1\right)\\
\leq&\left(1+\frac{2C_2}{R}\right)\left(\frac{n}{2}\left(\frac{(2k\beta+\beta')_+t_0}{4\beta(1-\beta)}+\frac1\beta\right)+C_5R^{-1}+3C_2R^{-1}T\right)\\
\leq&\left(1+\frac{2C_2}{R}\right)\left(t\psi_1(t)+C_5R^{-1}+3C_2R^{-1}T\right),\\
\end{split}
\end{equation}
where $C_4=\frac{C_1C_3}{4}$ and $C_5=\frac{nC_4}{2\min_{[0,T]}\beta}$. This implies that, when $x\in B_p(R)$ with $R$ sufficiently large,
\begin{equation}
F(x,t)\leq \left(1+\frac{2C_2}{R}\right)\left(t\psi_1(t)+C_5R^{-1}+3C_2R^{-1}T\right).
\end{equation}
Letting $R\to\infty$ in the last inequality, we complete the proof of the theorem.
\end{proof}
By choosing $\beta(t)=e^{-2\theta kt}$ with $\theta\in (0,1]$ in Theorem \ref{thm-LY-g}, we have the following extension of Hamilton's estimate \eqref{eq-Ha}.
\begin{cor}
  Let the notation be the same as in Theorem \ref{thm-LY-g} with $k>0$. Then,
\begin{equation}
e^{-2\theta kt}\|\nabla f\|^2-f_t\leq \frac{n}{2t}e^{2\theta kt}+\frac{nk(1-\theta)}{4(1-e^{-2\theta kt})}
\end{equation}
for any $\theta\in (0,1]$.
\end{cor}
\begin{proof}
It is clear that $\beta(t)=e^{-2\theta kt}$ satisfies (B1) and (B2). Note that
\begin{equation}\label{eqn-mono-1}
\frac{1}{\beta}+\frac{(2k\beta+\beta')_+t}{4\beta(1-\beta)}=e^{\theta x}+\frac{(1-\theta )x}{4(1-e^{-\theta x})}
\end{equation}
where $x=2kt$. Moreover
\begin{equation}\label{eq-mono-2}
\left(\frac{(1-\theta )x}{4(1-e^{-\theta x})}\right)'=\frac{(1-\theta)(1-e^{-y}-ye^{-y})}{4(1-e^{-y})^2}
\end{equation}
where $y=\theta x$. Note that
\begin{equation}
(1-e^{-y}-ye^{-y})'=ye^{-y}\geq 0
\end{equation}
where $y\geq 0$. Hence
\begin{equation}
1-e^{-y}-ye^{-y}\geq 0
\end{equation}
when $y\geq 0$. Then by \eqref{eqn-mono-1} and \eqref{eq-mono-2}, we know that $\frac{1}{\beta}+\frac{(2k\beta+\beta')_+t}{4\beta(1-\beta)}$ is increasing. So,
\begin{equation}
\psi_1(t)=\frac{n}{2t}\left(\frac{1}{\beta}+\frac{(2k\beta+\beta')_+t}{4\beta(1-\beta)}\right)=\frac{n}{2t}e^{2\theta kt}+\frac{nk(1-\theta)}{4(1-e^{-2\theta kt})}.
\end{equation}
This completes the proof of the corollary.
\end{proof}
By choosing $\beta(t)=\frac{1}{1+\theta k t}$ with $\theta>0$ in Theorem \ref{thm-LY-g}, we can obtain \eqref{eq-LY-linear}.
\begin{cor}
Let the notation be the same as in Theorem \ref{thm-LY-g} with $k>0$. Then, for any $\theta>0$,
\begin{equation}
  \frac{1}{1+\theta kt}\|\nabla f\|^2-f_t\leq \frac{n(1+\theta kt)}{2t}+\frac{n(2-\theta+2\theta kt)_+}{8\theta t}.
\end{equation}
\end{cor}
\begin{proof}Note that $\beta(t)=\frac{1}{1+\theta kt}$ with $\theta>0$ satisfies (B1) and (B2). Moreover,
\begin{equation}
\begin{split}
&\frac{1}{\beta}+\frac{(2k\beta+\beta)_+t}{4\beta(1-\beta)}\\
=&1+\theta k t+\frac{((2-\theta)+2\theta kt)_+}{4\theta}
\end{split}
\end{equation}
is increasing. So,
\begin{equation}
\begin{split}
\psi_1(t)=\frac{n}{2t}\left(\frac{1}{\beta}+\frac{(2k\beta+\beta')_+t}{4\beta(1-\beta)}\right)=\frac{n(1+\theta kt)}{2t}+\frac{n(2-\theta+2\theta kt)_+}{8\theta t}.
\end{split}
\end{equation}
This completes the proof of the corollary.
\end{proof}
Next, we give an upper bound of $\varphi_1(\beta,t)$ (see \eqref{eq-phi}) by using the test function $\beta(t)=1-\theta kt$.
\begin{cor}
Let the notation be the same as in Theorem \ref{thm-LY-g} with $k>0$. Then
\begin{equation}
\varphi_1(\beta,t)\leq\left\{\begin{array}{ll}\frac{n}{2\beta t}& t< \frac{1-\beta}{2k\beta}\\
\frac{3n}{8\beta t}+\frac{nk}{4(1-\beta)}&t\geq\frac{1-\beta}{2k\beta}
\end{array}\right.
\end{equation}
and hence, by Theorem \ref{thm-LY-g},
\begin{equation}
\beta\|\nabla f\|^2-f_t\leq\left\{\begin{array}{ll}\frac{n}{2\beta t}& t< \frac{1-\beta}{2k\beta}\\
\frac{3n}{8\beta t}+\frac{nk}{4(1-\beta)}&t\geq\frac{1-\beta}{2k\beta}
\end{array}\right.
\end{equation}
for any $\beta\in (0,1)$.
\end{cor}
\begin{proof}
For any given $\beta_0\in (0,1)$ and $t_0>0$, let $\theta_0>0$ be such that
\begin{equation}
\beta_0=1-\theta_0 kt_0.
\end{equation}
Let $\beta(t)=1-\theta_0 kt$. Note that
\begin{equation}
\begin{split}
&\frac{1}{\beta}+\frac{(2k\beta+\beta')_+t}{4\beta(1-\beta)}\\
=&\frac{1}{1-\theta_0 kt}+\max\left\{\frac{1}{2\theta_0}-\frac{1}{4(1-\theta_0 kt)},0\right\}\\
=&\max\left\{\frac{1}{2\theta_0}+\frac{3}{4(1-\theta_0 kt)},\frac{1}{1-\theta_0 kt}\right\}
\end{split}
\end{equation}
is increasing on $[0,t_0]$. So,
 \begin{equation}
\begin{split}
\varphi_1(\beta_0,t_0)\leq& \psi_1(t_0)=\frac{n}{2t_0}\left(\frac{1}{\beta_0}+\frac{(2k\beta_0+\beta')_+t_0}{4\beta_0(1-\beta_0)}\right)\\
=&\frac{n}{2\beta_0t_0}+\frac{n}{4(1-\beta_0)}\left(k-\frac{1-\beta_0}{2\beta_0t_0}\right)_+.
\end{split}
\end{equation}
This completes the proof of the corollary.
\end{proof}

Finally, we come to prove Theorem \ref{thm-LY-g-2}, Corollary \ref{cor-LY-g-2} and Corollary \ref{cor-LY-g-3}. The arguments are similar as before.
\begin{proof}[Proof of Theorem \ref{thm-LY-g-2}]The same as in the proof of Theorem \ref{thm-LY-g}, we only prove the complete noncompact case. Let $F$ and $G$ be the same as in the proof of Theorem \ref{thm-LY-g}.

 Note that the maximum of the function $h(Q)=\frac{aQ+1}{(1+bQ)^2}$ with $a, b>0$ and $Q\geq 0$ is
\begin{equation}\label{eq-max-h}
\max_{[0,\infty)} h=\left\{\begin{array}{ll}1&\frac ab<2\\\frac{a^2}{4(a-b)b}&\frac ab\geq 2.
\end{array}\right.
\end{equation}
Let
\begin{equation}
a_R(t)=\frac{[(2k\beta+\beta')_++C_1R^{-1}]t^2}{1+3C_1R^{-1}t}
\end{equation}
and
\begin{equation}
b(t)=(1-\beta)t,
\end{equation}
and
\begin{equation}
\sigma_R(t)=\max_{s\in [0,t]}\frac{a_R(s)}{b(s)}.
\end{equation}
Note that $\sigma_R(t)$ tends to $\sigma(t)$ as $R$ tending to $+\infty$. By applying \eqref{eq-max-h} to the right hand side of \eqref{eq-G-4} with $a=a_R(t_0)$ and $b=b(t_0)$, we have
\begin{equation}\label{eq-G-5}
\begin{split}
G(x_0,t_0)\leq& \left\{\begin{array}{ll}\left(1+2C_2R^{-1}\right)(1+3C_1R^{-1}t_0)\frac{n}{2\beta(t_0)}&\sigma(t)<2\\
\left(1+2C_2R^{-1}\right)(1+3C_1R^{-1}t_0)\frac{n\sigma_{R}^2(t)}{8(\sigma_R(t)-1)\beta(t_0)}&\sigma(t)\geq 2\end{array}\right.\\
\leq&\left\{\begin{array}{ll}\left(1+2C_2R^{-1}\right)(1+3C_1R^{-1}t)\frac{n}{2\lambda(t)}&\sigma(t)<2\\
\left(1+2C_2R^{-1}\right)(1+3C_1R^{-1}t)\frac{n\sigma_{R}^2(t)}{8(\sigma_R(t)-1)\lambda(t)}&\sigma(t)\geq 2\end{array}\right.\\
\end{split}
\end{equation}
when $R$ is sufficiently large. This implies that, when $x\in B_p(R)$ with $R$ sufficiently large,
\begin{equation}
F\leq\left\{\begin{array}{ll}\left(1+2C_2R^{-1}\right)(1+3C_1R^{-1}t)\frac{n}{2\lambda(t)}&\sigma(t)<2\\
\left(1+2C_2R^{-1}\right)(1+3C_1R^{-1}t)\frac{n\sigma_{R}^2(t)}{8(\sigma_R(t)-1)\lambda(t)}&\sigma(t)\geq 2\end{array}\right.\\
\end{equation}
Letting $R\to\infty$ in the last inequality, we complete the proof of the theorem.
\end{proof}
\begin{proof}[Proof of Corollary \ref{cor-LY-g-2}] When $\beta$ is a constant in
$(0,1)$, it is clear that $\sigma(t)=\frac{2k\beta t}{1-\beta}$ and $\lambda(t)=\beta$. So, by Theorem \ref{thm-LY-g-2}, when $t<\frac{1-\beta}{k\beta}$, that is, $\sigma(t)<2$, we have
\begin{equation}
\beta\|\nabla f\|^2-f_t\leq \frac{n}{2\beta t}.
\end{equation}
Moreover, when $t\geq\frac{1-\beta}{k\beta}$, we have
\begin{equation}
\begin{split}
\beta\|\nabla f\|^2-f_t\leq&\frac{n\left(\frac{2k\beta t}{1-\beta}\right)^2}{8(\frac{2k\beta t}{1-\beta}-1)\beta t}\\
=&\frac{nk}{4(1-\beta)}\cdot\frac{1}{1-\frac{1-\beta}{2k\beta t}}\\
=&\frac{nk}{4(1-\beta)}\left(1+\frac{1-\beta}{2k\beta t}+\left(\frac{1-\beta}{2k\beta t}\right)^2+\cdots\right)\\
\leq&\frac{nk}{4(1-\beta)}\left(1+\frac{1-\beta}{k\beta t}\right)\\
=&\frac{n}{4\beta t}+\frac{nk}{4(1-\beta)}
\end{split}
\end{equation}
by noting that $\frac{1-\beta}{2k\beta t}\leq\frac{1}{2}$.
\end{proof}
\begin{proof}[Proof of Corollary \ref{cor-LY-g-3}] For $\beta_0\in (0,1)$ and $t_0\in (0,T]$, let $\theta_0>0$ be such that
\begin{equation}
\beta_0=1-\theta_0kt_0.
\end{equation}
Let $\beta(t)=1-\theta_0 kt$. It is clear that $\lambda(t)=1-\beta_0 kt$. By direct computation, one has
\begin{equation}
\sigma(t)=\left(\frac{2}{\theta_0}-1\right)_+
\end{equation}
When $t_0<\frac{3(1-\beta_0)}{2k}$, we have $\theta_0>\frac{2}{3}$, so $\sigma(t)<2$ for any $t\in [0,t_0]$. By Theorem \ref{thm-LY-g-2},
\begin{equation}
\psi_2(t_0)\leq \frac{n}{2\beta_0t_0}.
\end{equation}
On the other hand, when $t_0\geq\frac{3(1-\beta_0)}{2k}$, that is $\theta_0\leq \frac{2}3$, we have $\sigma(t)\geq2$ for any $t\in [0,t_0]$. By Theorem \ref{thm-LY-g-2},
\begin{equation}
\begin{split}
\psi_2(t_0)\leq& \frac{n(\frac{2}{\theta_0}-1)^2}{8(\frac{2}{\theta_0}-2)\beta_0t_0}\\
=&\frac{nk}{16\beta_0(1-\beta_0)}\frac{(2-\theta_0)^2}{1-\theta_0}\\
=&\frac{nk}{16\beta_0(1-\beta_0)}\left(4+\theta_0^2+\theta_0^3+\cdots\right)\\
\leq&\frac{nk}{16\beta_0(1-\beta_0)}\left(4+3\theta_0^2\right)\\
=&\frac{3(1-\beta_0)n}{16k\beta_0 t_0^2}+\frac{nk}{4(1-\beta_0)\beta_0}
\end{split}
\end{equation}
by substituting $\theta_0=\frac{1-\beta_0}{kt_0}$  into the inequality and noting that $\theta_0\leq \frac{2}{3}$.
\end{proof}

\end{document}